\documentclass[a4paper,11pt,twoside,notitlepage]{amsart}

\usepackage{amsmath,amssymb,amsthm,mathtools,mathrsfs}
\usepackage[shortlabels]{enumitem}
\usepackage[ocgcolorlinks,linkcolor=blue,citecolor=blue,urlcolor=blue]{hyperref}
\usepackage{url}
\numberwithin{equation}{section}
\mathtoolsset{showonlyrefs}
\allowdisplaybreaks

\newtheorem{theorem}{Theorem}[section]
\newtheorem{proposition}[theorem]{Proposition}
\newtheorem{lemma}[theorem]{Lemma}
\newtheorem{corollary}[theorem]{Corollary}
\newtheorem{definition}[theorem]{Definition}
\newtheorem{remark}[theorem]{Remark}
\newtheorem{question}{Question}[section]

\DeclareMathOperator{\diam}{diam}
\DeclareMathOperator{\Id}{Id}

\newcommand{\R}{\mathbb R}

\newcommand{\eps}{\varepsilon}

\newcommand{\Dinfty}{\Delta_\infty}
\newcommand{\para}[1]{\vspace{3mm}\noindent\textbf{#1.}}
\newcommand{\Xray}{\mathcal X}

\title[Calder\'on problem for the infinity Laplacian]{The Calder\'on problem for the infinity Laplacian with large affine boundary data}

\author[Y.-H. Lin]{Yi-Hsuan Lin}
\address{Department of Applied Mathematics, National Yang Ming Chiao Tung University, Hsinchu, Taiwan \& Fakult\"at f\"ur Mathematik, University of Duisburg-Essen, Essen, Germany}
\email{yihsuanlin3@gmail.com}


\hypersetup{
	pdftitle={The Calderon problem for the infinity Laplacian with large affine boundary data},
	pdfauthor={Yi-Hsuan Lin},
	pdfsubject={The Calderon problem for a quasilinear degenerate elliptic infinity Laplacian equation},
	pdfkeywords={Calderon problem, infinity Laplacian, quasilinear degenerate elliptic equation, Dirichlet-to-Neumann map, viscosity solutions, X-ray transform}
}

\begin{document}

\begin{abstract}
	We study the Calder\'on problem for the equation $-\Delta_\infty u+q(x)u=0$ in a bounded convex domain with $C^2$ boundary. We prove that a positive potential $q$ is uniquely determined and explicitly reconstructible from the measurements $\Lambda_q\bigl(t(e\cdot x+b)|_{\partial\Omega}\bigr)$, where $e\in\mathbb S^{n-1}$, the offset $b>\sup_{\overline\Omega}|x|$ is fixed, and $t\to\infty$. The first potential-dependent term in the large amplitude asymptotics determines weighted integrals of $q$ over the chords parallel to $e$. Combining the measurements in the directions $e$ and $-e$ gives the X-ray transform of the zero extension of $q$, and the Fourier slice identity yields reconstruction and uniqueness.
\end{abstract}

\maketitle
\tableofcontents

\section{Introduction}

The infinity Laplacian is the second-order operator
\begin{equation}\label{eq:Dinfty-intro}
	\Dinfty u=\langle D^2u\,Du,Du\rangle=\sum_{i,j=1}^n u_{x_i}u_{x_j}u_{x_ix_j}.
\end{equation}
For the equation considered in this paper, define
\begin{equation}\label{eq:F-intro}
	F_q(x,r,p,X)=-p^TXp+q(x)r.
\end{equation}
If $X\le Y$ in the order of symmetric matrices, then $F_q(x,r,p,X)\ge F_q(x,r,p,Y)$. Thus, $F_q$ is degenerate elliptic in the viscosity convention. Its principal coefficient matrix is $p\otimes p$. When $p\ne0$, this matrix has eigenvalues $|p|^2,0,\ldots,0$, while it vanishes when $p=0$. The equation is therefore not uniformly elliptic. Since \eqref{eq:F-intro} is linear in the Hessian variable, with coefficients depending quadratically on the gradient, it is a quasilinear equation of nondivergence form. The zeroth-order term does not affect ellipticity. If $q\ge c_0>0$, then $F_q$ is strictly increasing in the solution variable, which gives the comparison principle used below.

The infinity Laplacian arose in Aronsson's study of absolutely minimizing Lipschitz extensions and supremal variational problems \cite{Aronsson1967,Aronsson1968}. The homogeneous equation $\Dinfty u=0$ is naturally understood in the viscosity sense. Its relation to absolutely minimizing functions, comparison with cones, and limits of $p$-harmonic equations was developed in \cite{Jensen1993,CrandallEvansGariepy2001,AronssonCrandallJuutinen2004}; see also \cite{PeresSchrammSheffieldWilson2009,Lindqvist2014}. At every point where $Du\ne0$, writing $\widehat{Du}=Du/|Du|$, one has $\Dinfty u=|Du|^2D^2u[\widehat{Du},\widehat{Du}]$. Thus, the operator records the second derivative only in the gradient direction. This rank-one structure is the main feature used in the inverse problem.

The infinity Laplacian also appears as a nonlinear interpolation operator. The absolutely minimizing Lipschitz extension selects an extension by controlling its largest local slope, a principle used in image interpolation and restoration \cite{CasellesMorelSbert1998}. On weighted graphs, the same principle leads to Lipschitz learning for semi-supervised classification, whose continuum limit is governed by infinity Laplace type equations \cite{Calder2019Lipschitz}. The tug-of-war representation gives a stochastic dynamic-programming interpretation of the equation \cite{PeresSchrammSheffieldWilson2009}. In the problem considered here, $q$ is a spatially varying lower-order coefficient, and the question is whether it can be determined from boundary inputs and the corresponding normal derivatives.

Let $\Omega\subset\R^n$ be a bounded domain, let $q$ be a zeroth-order coefficient, and let $f\in C(\partial\Omega)$. We consider
\begin{equation}\label{eq:forward-intro}
	\begin{cases}
		-\Dinfty u+q(x)u=0 & \text{in }\Omega,\\
		u=f & \text{on }\partial\Omega.
	\end{cases}
\end{equation}
If $q\in C(\overline\Omega)$ and $q\ge c_0>0$, then \eqref{eq:forward-intro} has a unique viscosity solution $u_f^q\in C(\overline\Omega)\cap\operatorname{Lip}_{\mathrm{loc}}(\Omega)$. The boundary condition is attained pointwise, while the local Lipschitz regularity follows by writing the equation as $\Dinfty u_f^q=q u_f^q$ and applying the regularity theory for bounded inhomogeneities. The viscosity formulation itself only requires continuity, and the interior Lipschitz regularity does not by itself provide a normal derivative on $\partial\Omega$. A separate boundary regularity result is therefore needed in order to define the boundary measurements.

Suppose that $\partial\Omega$ is of class $C^1$ and that $f\in C^1(\partial\Omega)$. Since $u_f^q$ satisfies $\Dinfty u_f^q=h$ with $h(x)=q(x)u_f^q(x)$ bounded and continuous, the boundary differentiability theorem for inhomogeneous infinity Laplace equations implies that $u_f^q$ is differentiable at every point $x_0\in\partial\Omega$. More precisely, there exists a vector $Du_f^q(x_0)\in\R^n$ such that
\[
u_f^q(x)=u_f^q(x_0)+Du_f^q(x_0)\cdot(x-x_0)+o(|x-x_0|)
\]
as $x\to x_0$ within $\overline\Omega$. In particular, the outward normal derivative $\partial_\nu u_f^q(x_0)=Du_f^q(x_0)\cdot\nu(x_0)$ is well defined at every boundary point. We may therefore define the DN map by
\begin{equation}\label{eq:DN-intro}
	\Lambda_q(f)(x)=\partial_\nu u_f^q(x),\quad x\in\partial\Omega,
\end{equation}
where $\nu$ denotes the outward unit normal. The same boundary differentiability will be used later to convert the linear vanishing of the chord-barrier error at the endpoints into the large amplitude asymptotics of the normal derivatives.

The Calder\'on problem is to determine $q$ from $\Lambda_q$. We show that the values of $\Lambda_q(f)$ for arbitrary boundary functions are not needed. Let $R_\Omega=\sup_{x\in\overline\Omega}|x|$, fix $b>R_\Omega$, and set $\ell_{e,b}(x)=e\cdot x+b$ for $e\in\mathbb S^{n-1}$. Then $\ell_{e,b}>0$ on $\overline\Omega$. The measurements used in this paper are $\Lambda_q\bigl(t\ell_{e,b}|_{\partial\Omega}\bigr)$, where $e\in\mathbb S^{n-1}$ and $t$ is sufficiently large. For each fixed $e$, the boundary values therefore lie on the one-parameter ray $\{t\ell_{e,b}|_{\partial\Omega}:t>0\}$.

\begin{question}\label{Q:inverse}
	Do the measurements $\Lambda_q\bigl(t\ell_{e,b}|_{\partial\Omega}\bigr)$, with $e\in\mathbb S^{n-1}$ and $t$ sufficiently large, uniquely determine $q$?
\end{question}

Our main result gives an affirmative answer. The first potential-dependent term in the large amplitude measurements determines the X-ray transform of the zero extension of $q$, from which both uniqueness and reconstruction follow.

\para{Earlier literature}
The classical Calder\'on problem asks whether an isotropic conductivity can be recovered from its Dirichlet-to-Neumann map \cite{calderon}. In dimensions $n\ge3$, global uniqueness was proved by Sylvester and Uhlmann using complex geometric optics solutions \cite{sylvester1987global}. Their work initiated a broad development of inverse boundary value problems for linear elliptic equations. For nonlinear equations, the dependence of the boundary measurements on the boundary values provides information that is absent from a single linearized equation.

A standard approach to nonlinear inverse problems is to differentiate the nonlinear Dirichlet-to-Neumann map with respect to the boundary data. Early applications include Isakov's work on semilinear parabolic equations \cite{Isa93}, the determination of nonlinear conductivities by Kang and Nakamura \cite{KN02}, and the elliptic results of Sun \cite{Sun96,Sun10}. The higher-order linearization method, based on interactions of several linearized solutions, was introduced in a geometric hyperbolic setting in \cite{KLU2018} and later developed for semilinear elliptic equations in \cite{FO20,LLLS2019nonlinear}. Related results include partial data problems and nonlinearities with noninteger powers \cite{LLLS2019partial,LLST2022inverse}.

Quasilinear elliptic inverse problems have been studied using first and higher-order linearization, complex geometric optics solutions, and boundary determination; see \cite{CNV2019reconstruction,CFKKU2021calderon,KKU2022partial} and the references therein. Geometric quasilinear problems include the recovery of Riemannian metrics and coefficients from minimal area or minimal surface data \cite{ABN20,CLLT24,Nur24}. We refer to \cite{lassas2025introduction} for a broader introduction to nonlinear inverse problems.

The argument is also related to inverse source problems. For a linear equation, a compactly supported perturbation of a solution may change the source without changing the boundary Cauchy data. Nonlinear effects can remove this obstruction, as shown for semilinear elliptic and reaction-diffusion equations in \cite{LL24_elliptic_source,KLL_reaction_source}. More recent results concern inverse source problems for the Monge--Amp\`ere equation and for broader classes of admissible fully nonlinear equations \cite{LL2025IP_Monge_Ampere,LLW26_fully_nonlinear,Lin2026gauge}. Large boundary data give a mechanism different from higher-order linearization. In several fully nonlinear problems, a rank-deficient leading profile forces the first source-dependent correction to satisfy lower-dimensional equations, leading to chord or affine section tomography \cite{CG26_large,Lin2026KHessian}. The Hessian equations in those works are fully nonlinear in the second derivatives. The equation considered here is quasilinear and degenerate elliptic, but its large affine limit produces a related rank-one equation.

The forward theory for inhomogeneous infinity Laplace equations was developed in \cite{LuWang2008,BhattacharyaMohammed2012}. The strong degeneracy of the operator makes regularity substantially different from the uniformly elliptic theory; characteristic and adjoint methods were developed in \cite{EvansYu2005,EvansSmart2011Differentiability,EvansSmart2011Adjoint}. The boundary differentiability theorem used in \eqref{eq:DN-intro} is due to Hong \cite{Hong2014}. Inverse boundary value problems for the finite $p$-Laplacian have also been studied, for example in \cite{SaloZhong2012}. The limit $p\to\infty$ does not directly give the argument used here, since the first correction around an affine infinity harmonic function is governed by a rank-one equation rather than a uniformly elliptic one.

\para{Difficulties and ideas of the proof}
The proof has three main difficulties. First, solutions of \eqref{eq:forward-intro} are viscosity solutions, and the equation cannot be differentiated through a formal classical expansion. Second, the infinity Laplacian is not in divergence form, so the usual Alessandrini identity is not available. Third, the equation governing the first large amplitude correction contains second derivatives in only one direction.

The different homogeneities of the two terms provide the starting point. Since $\Dinfty(tu)=t^3\Dinfty u$, uniqueness for the forward problem gives the exact identity $u_{tf}^q=t u_f^{q/t^2}$. Thus, increasing the amplitude of the boundary value for a fixed potential is equivalent to keeping the boundary value fixed and replacing $q$ by $t^{-2}q$. This reduction is specific to the linear potential $q(x)u$. For the equation with $q(x)u^3$, both terms have degree three under $u\mapsto tu$, so the same normalization leaves the potential unchanged. 

We apply this identity to the affine infinity harmonic function
\begin{equation}\label{eq:affine_BC}
	\ell_{e,b}(x)=e\cdot x+b.
\end{equation} 
Let $\eps=t^{-2}$ and set $v_\eps=t^{-1}u_{t\ell_{e,b}}^q$, then $v_\eps$ is the unique viscosity solution of
\begin{equation}\label{eq:rescaled-intro}
	\begin{cases}
		-\Dinfty v_\eps+\eps q(x)v_\eps=0 & \text{in }\Omega,\\
		v_\eps=\ell_{e,b} & \text{on }\partial\Omega.
	\end{cases}
\end{equation}
Define
\begin{equation}\label{eq:intro-expansion}
	z_\eps=\frac{v_\eps-\ell_{e,b}}{\eps},\quad\text{equivalently}\quad \frac{u_{t\ell_{e,b}}^q}{t}=\ell_{e,b}+\eps z_\eps.
\end{equation}
This is an exact decomposition. Since $v_\eps$ and $\ell_{e,b}$ have the same boundary trace, $z_\eps=0$ on $\partial\Omega$.

Substituting $v_\eps=\ell_{e,b}+\eps z_\eps$ into \eqref{eq:rescaled-intro} and dividing by $\eps$ gives
\begin{equation}\label{eq:z-exact-intro}
	-(e+\eps Dz_\eps)^TD^2z_\eps(e+\eps Dz_\eps)+q(x)(\ell_{e,b}+\eps z_\eps)=0\quad\text{in }\Omega,
\end{equation}
where the equation is understood in the viscosity sense. Lemma~\ref{lem:z-equation} proves this change of variables directly with test functions. No differentiability of the solution with respect to $\eps$ is assumed.

The limiting equation is
\begin{equation}\label{eq:rank-one-intro}
	-\partial_e^2w+q(x)(e\cdot x+b)=0,
\end{equation}
where $\partial_e^2w=D^2w[e,e]$. On each line parallel to $e$, this is a one-dimensional equation. We define $w$ by solving \eqref{eq:rank-one-intro} on each chord of $\Omega$ with zero endpoint values and prove, by explicit barriers, that $z_\eps$ converges to $w$ on every fixed positive length chord.

The exact identity \eqref{eq:intro-expansion} can also be written as $u_{t\ell_{e,b}}^q=t\ell_{e,b}+t^{-1}z_{t^{-2}}$. Taking the normal derivative gives
\[
\Lambda_q\bigl(t\ell_{e,b}|_{\partial\Omega}\bigr)=t\,e\cdot\nu+t^{-1}\partial_\nu z_{t^{-2}}.
\]
Thus, the first potential-dependent term in the DN measurements is determined by the boundary behavior of $z_\eps$, and it is necessary to control the normal derivatives of $z_\eps$, rather than only its values in the interior.

To obtain this boundary information, we use the chord notation introduced below. For $y\in Y_e$, the line $y+\R e$ meets $\Omega$ in the interval $\{y+se:\alpha_e(y)<s<\beta_e(y)\}$. A global slab barrier first gives an $\eps$-independent bound for $z_\eps$; see Proposition~\ref{prop:global-bound}. For a fixed $y_0\in Y_e$, Lemma~\ref{lem:tube-barrier} constructs local barriers in a small tube around the chord through $y_0$ and proves the central chord estimate \eqref{eq:central-chord-estimate}. The factor on the right-hand side of that estimate vanishes linearly at both endpoints of the chord.

Proposition~\ref{prop:endpoint-derivative} applies one-sided difference quotients to \eqref{eq:central-chord-estimate} and obtains an $O(\eps)$ estimate for the directional derivatives $\partial_ez_\eps-\partial_ew_{q,e,b}$ at the two endpoints. Since both functions vanish on the corresponding boundary faces, their gradients are normal to $\partial\Omega$ there. At a non-glancing endpoint one has $e\cdot\nu\ne0$, so the identity $\partial_e=(e\cdot\nu)\partial_\nu$ converts the directional derivative estimate into an estimate for $\partial_\nu z_\eps-\partial_\nu w_{q,e,b}$. Proposition~\ref{prop:local-uniformity} makes this estimate uniform on compact subsets of $\Gamma_e$. The argument does not require an interior $C^1$ estimate or a nonvanishing condition on the gradient.

Finally, integrating \eqref{eq:rank-one-intro} along a chord expresses a weighted integral of $q$ in terms of the two endpoint values of $\partial_\nu w$. Repeating the construction in the direction $-e$ removes the affine weight. The resulting quantities give the X-ray transform of the zero extension of $q$, and the Fourier slice identity yields reconstruction and uniqueness.

\para{Mathematical formulation and main results}
Throughout the paper, $n\ge2$, $\Omega\subset\R^n$ is a bounded convex domain with $C^2$ boundary, and
\begin{equation}\label{eq:q-assumptions-intro}
	q\in C^2(\overline\Omega),\quad q(x)\ge c_0>0.
\end{equation}
The positivity of $q$ gives comparison and uniqueness for the forward problem. The $C^2$ regularity is used in the local chord construction, while the forward theory requires less.

For $e\in\mathbb S^{n-1}$, let $P_e=\Id-e\otimes e$ and $Y_e=P_e\Omega\subset e^\perp$. For each $y\in Y_e$, convexity gives numbers $\alpha_e(y)<\beta_e(y)$ such that
\begin{equation}\label{eq:chord-intro}
	\Omega\cap(y+\R e)=\{y+se:\, \alpha_e(y)<s<\beta_e(y)\}.
\end{equation}
We write $x_e^-(y)=y+\alpha_e(y)e$ and $x_e^+(y)=y+\beta_e(y)e$. These endpoints are non-glancing: $e\cdot\nu(x_e^+(y))>0$ and $e\cdot\nu(x_e^-(y))<0$. We also set
\begin{equation}\label{eq:Gamma-intro}
	\Gamma_e=\{x\in\partial\Omega:\, e\cdot\nu(x)\ne0\}.
\end{equation}
The first theorem identifies the first potential-dependent term in the large amplitude measurements.

\begin{theorem}[large amplitude DN asymptotics]\label{thm:asymp-intro}
	Assume \eqref{eq:q-assumptions-intro} and fix $b>R_\Omega$. For $e\in\mathbb S^{n-1}$ and $t>0$, let $u_{t,e,b}^q\in C(\overline{\Omega})$ be the unique viscosity solution of
	\begin{equation}\label{eq:large-forward-intro}
		\begin{cases}
			-\Dinfty u_{t,e,b}^q+q(x)u_{t,e,b}^q=0 & \text{in }\Omega,\\
			u_{t,e,b}^q=t\ell_{e,b} & \text{on }\partial\Omega,
		\end{cases}
	\end{equation}
	where $\ell_{e,b}$ is the affine function defined by \eqref{eq:affine_BC}. Then, for every $x\in\Gamma_e$, the limit
	\begin{equation}\label{eq:H-intro}
		H_q(e,b;x)=\lim_{t\to\infty}t\left[\Lambda_q\bigl(t\ell_{e,b}|_{\partial\Omega}\bigr)(x)-t\,e\cdot\nu(x)\right]
	\end{equation}
	exists, locally uniformly on $\Gamma_e$.
	
	For each $y\in Y_e$, let $w_{q,e,b}$ be the unique solution of
	\begin{equation}\label{eq:w-intro}
		\begin{cases}
			\partial_s^2w_{q,e,b}(y,s)=q(y+se)(s+b) & \text{for }\alpha_e(y)<s<\beta_e(y),\\
			w_{q,e,b}(y,s)=0 & \text{for }s=\alpha_e(y)\text{ and }s=\beta_e(y).
		\end{cases}
	\end{equation}
	The function $w_{q,e,b}$ is of class $C^2$ on every closed chord tube whose transverse base is compactly contained in $Y_e$. In particular, it is $C^2$ up to every point of $\Gamma_e$, and
	\begin{equation}\label{eq:H-equals-normal-intro}
		H_q(e,b;x)=\partial_\nu w_{q,e,b}(x),\quad x\in\Gamma_e.
	\end{equation}
	Moreover, for every compact set $K\Subset\Gamma_e$, there are constants $C_K>0$ and $t_K>0$ such that
	\begin{equation}\label{eq:DN-rate-intro}
		\sup_{x\in K}\big|\Lambda_q\bigl(t\ell_{e,b}|_{\partial\Omega}\bigr)(x)-t\,e\cdot\nu(x)-t^{-1}\partial_\nu w_{q,e,b}(x)\big|\le C_Kt^{-3}
	\end{equation}
	for $t\ge t_K$.
\end{theorem}

\begin{theorem}[Restricted data uniqueness]\label{thm:main}
	Let $q_1,q_2\in C^2(\overline\Omega)$ satisfy $q_j\ge c_0>0$, and fix $b>R_\Omega$. Suppose that, for every $e\in\mathbb S^{n-1}$, there is $t_e>0$ such that
	\begin{equation}\label{eq:DN-equality-intro}
		\Lambda_{q_1}\bigl(t\ell_{e,b}|_{\partial\Omega}\bigr)=\Lambda_{q_2}\bigl(t\ell_{e,b}|_{\partial\Omega}\bigr)\quad\text{on }\partial\Omega
	\end{equation}
	for all $t\ge t_e$. Then $q_1=q_2$ in $\Omega$.
\end{theorem}

We next describe the reconstruction formula provided by the proof. Let $Q=q\mathbf 1_\Omega$ be the zero extension of $q$ to $\R^n$. For $e\in\mathbb S^{n-1}$ and $y\in e^\perp$, define the \emph{X-ray transform} of $Q$ by
\begin{equation}\label{eq:Xray-intro}
	\Xray Q(e,y):=\int_\R Q(y+se)\,ds.
\end{equation}
If $y\in Y_e$, this is the integral of $q$ over the chord $\Omega\cap(y+\R e)$. If $y\notin Y_e$, the line $y+\R e$ does not meet the open set $\Omega$, and hence, $\Xray Q(e,y)=0$.

We define
\begin{equation}\label{eq:I-intro}
		I_q(e,b;y):=\begin{cases}
		(e\cdot\nu(x_e^+(y)))H_q(e,b;x_e^+(y))-(e\cdot\nu(x_e^-(y)))H_q(e,b;x_e^-(y)) &\text{ for }y\in Y_e, \\
		0 &\text{ for }y\notin Y_e.
		\end{cases}
\end{equation}
To see how the endpoint data determine a chord integral, fix $y\in Y_e$ and write $\alpha=\alpha_e(y)$, $\beta=\beta_e(y)$, $x_-=y+\alpha e$, and $x_+=y+\beta e$. By \eqref{eq:w-intro}, the function $W_y(s)=w_{q,e,b}(y+se)$ satisfies
\[
W_y''(s)=q(y+se)(s+b),\quad \alpha<s<\beta.
\]
Integrating this equation gives
\[
\int_\alpha^\beta q(y+se)(s+b)\,ds=W_y'(\beta)-W_y'(\alpha)=\partial_e w_{q,e,b}(x_+)-\partial_e w_{q,e,b}(x_-).
\]
Since $w_{q,e,b}$ vanishes on $\partial\Omega$ near both endpoints, its tangential derivatives vanish there, then we have $\partial_e w_{q,e,b}(x_\pm)=(e\cdot\nu(x_\pm))\partial_\nu w_{q,e,b}(x_\pm)$. By \eqref{eq:H-equals-normal-intro} in Theorem~\ref{thm:asymp-intro}, the definition of $I_q$ gives $I_q(e,b;y)=\int_{\alpha_e(y)}^{\beta_e(y)}q(y+se)(s+b)\,ds$. The same chord can be parametrized in the opposite direction $-e$. Since $P_{-e}=P_e$, the transverse parameter $y$ remains unchanged, while $\alpha_{-e}(y)=-\beta_e(y)$ and $\beta_{-e}(y)=-\alpha_e(y)$. Applying the preceding identity in the direction $-e$ and reversing the chord parameter gives $I_q(-e,b;y)=\int_{\alpha_e(y)}^{\beta_e(y)}q(y+se)(b-s)\,ds$. Adding this identity to the corresponding formula for $I_q(e,b;y)$ cancels the term linear in $s$ and gives
\[
I_q(e,b;y)+I_q(-e,b;y)
=2b\int_{\alpha_e(y)}^{\beta_e(y)}q(y+se)\,ds.
\]
The integral on the right-hand side is the X-ray transform of the zero extension $Q=q\mathbf 1_\Omega$ along the line $y+\R e$. Thus, the large amplitude DN measurements determine $\Xray Q(e,y)$ for every $e\in\mathbb S^{n-1}$ and $y\in e^\perp$.

The Fourier slice identity then recovers the Fourier transform of $Q$. More precisely, if $\xi\in e^\perp$, then
\[
\widehat Q(\xi)
=\int_{e^\perp}e^{-\mathsf i y\cdot\xi}\Xray Q(e,y)\,dy.
\]
For every $\xi\ne0$, one may choose a direction $e_\xi\in\mathbb S^{n-1}$ satisfying $e_\xi\cdot\xi=0$. Substitution of the recovered X-ray transform into the Fourier slice identity therefore determines $\widehat Q(\xi)$. We write $\mathcal F^{-1}$ for the inverse Fourier transform.

\begin{corollary}[Reconstruction formula]\label{cor:reconstruction-intro}
	Under the assumptions of Theorem~\ref{thm:asymp-intro}, one has
	\begin{equation}\label{eq:Xray-reconstruction-intro}
		\Xray Q(e,y)=\frac{I_q(e,b;y)+I_q(-e,b;y)}{2b}
	\end{equation}
	for every $e\in\mathbb S^{n-1}$ and $y\in e^\perp$.
	
	For $\xi\ne0$, choose any $e_\xi\in\mathbb S^{n-1}$ satisfying $e_\xi\cdot\xi=0$. With the Fourier transform convention $\widehat Q(\xi)=\int_{\R^n}e^{-\mathsf i x\cdot\xi}Q(x)\,dx$, one has
	\begin{equation}\label{eq:Fourier-reconstruction-intro}
		\widehat Q(\xi)=\frac1{2b}\int_{e_\xi^\perp}e^{-\mathsf i y\cdot\xi}\left[I_q(e_\xi,b;y)+I_q(-e_\xi,b;y)\right]\,dy.
	\end{equation}
	The value $\widehat Q(0)$ is determined by continuity of the Fourier transform of the compactly supported function $Q\in L^1(\R^n)$. Consequently, $Q=\mathcal F^{-1}\widehat Q$ in $\mathcal S'(\R^n)$, and $q=Q|_\Omega$.
\end{corollary}

Two features of the result are worth noting. For each fixed direction $e$, the offset $b$ is fixed and only the amplitude $t$ varies. Thus, the boundary values belong to the one-parameter family $t\ell_{e,b}|_{\partial\Omega}$. The transverse variable $y\in e^\perp$ is not an additional boundary input; it specifies the chord whose two endpoints are used to evaluate the pointwise DN measurements. The proof only uses the large amplitude coefficient $H_q$, and hence equality of the measurements along an unbounded sequence of amplitudes is sufficient; see Remark~\ref{rem:weaker-data}.

Convexity is used to ensure that each intersection $\Omega\cap(y+\R e)$ is a single interval and that the endpoints of every positive length chord are non-glancing. These properties allow the correction equation to be solved chordwise and the endpoint directional derivatives to be converted into normal derivatives. No positive lower bound for the curvature of $\partial\Omega$ is needed. In particular, the domain need not be strictly or uniformly convex.

\para{Organization of the paper}
Section~\ref{sec:forward} establishes the forward well-posedness and defines the DN map. Section~\ref{sec:geometry} develops the geometry of parallel chords and constructs the rank-one correction. Section~\ref{sec:global-asymptotics} derives the exact equation for the difference quotient and proves the global slab estimate. Section~\ref{sec:boundary-asymptotics} establishes the local chord-tube estimate and the resulting DN asymptotics. Section~\ref{sec:recovery} recovers the X-ray transform and proves the reconstruction and uniqueness results.

\section{The forward problem and the Dirichlet-to-Neumann map}\label{sec:forward}

For $q\in C(\overline\Omega)$, recall the operator $F_q$ is given by \eqref{eq:F-intro}.

\begin{proposition}[Degenerate ellipticity and quasilinear structure]\label{prop:PDE-structure}
	For every $(x,r,p)$, the map $X\mapsto F_q(x,r,p,X)$ is degenerate elliptic. More precisely, if $X\le Y$, then
	\begin{equation}\label{eq:degenerate-ellipticity}
		F_q(x,r,p,X)-F_q(x,r,p,Y)=p^T(Y-X)p\ge0.
	\end{equation}
	Written as $-a^{ij}(Du)u_{ij}+q(x)u=0$, the principal coefficient matrix is $a(Du)=Du\otimes Du$. If $Du\ne0$, this matrix has rank one and eigenvalues $|Du|^2,0,\ldots,0$; if $Du=0$, it vanishes. Hence, the equation is degenerate elliptic, but not uniformly elliptic. It is quasilinear and of nondivergence form. If $q\ge c_0>0$ and $r_2\ge r_1$, then $F_q(x,r_2,p,X)-F_q(x,r_1,p,X)\ge c_0(r_2-r_1)$.
\end{proposition}

\begin{proof}
	Identity \eqref{eq:degenerate-ellipticity} follows from $Y-X\ge0$. The spectral statement follows from $(p\otimes p)p=|p|^2p$ and $(p\otimes p)\xi=0$ whenever $\xi\perp p$. The remaining claims follow from $\Dinfty u=(Du\otimes Du):D^2u$ and the lower bound for $q$.
\end{proof}

\begin{definition}[Viscosity solution]\label{def:viscosity}
	A function $u\in C(\Omega)$ is called a viscosity subsolution of $F_q(x,u,Du,D^2u)=0$ in $\Omega$ if, whenever $\phi\in C^2(\Omega)$ and $u-\phi$ has a local maximum at $x_0\in\Omega$, one has $F_q(x_0,u(x_0),D\phi(x_0),D^2\phi(x_0))\le0$. It is a viscosity supersolution if, whenever $\phi\in C^2(\Omega)$ and $u-\phi$ has a local minimum at $x_0\in\Omega$, one has $F_q(x_0,u(x_0),D\phi(x_0),D^2\phi(x_0))\ge0$. A viscosity solution is both a viscosity subsolution and a viscosity supersolution.
	
	Given $f\in C(\partial\Omega)$, a function $u\in C(\overline\Omega)$ is a viscosity solution of
	\begin{equation}\label{eq:forward-problem}
		\begin{cases}
			-\Dinfty u+q(x)u=0 & \text{in }\Omega,\\
			u=f & \text{on }\partial\Omega
		\end{cases}
	\end{equation}
	if it is a viscosity solution of $F_q(x,u,Du,D^2u)=0$ in $\Omega$ and satisfies $u=f$ pointwise on $\partial\Omega$.
\end{definition}

Since the unnormalized infinity Laplacian is continuous at $p=0$, no separate convention is needed at critical points. Definition~\ref{def:viscosity} is the usual viscosity definition for a continuous degenerate elliptic operator \cite{CrandallIshiiLions1992}.

\begin{proposition}[Comparison principle]\label{prop:comparison}
	Let $\Omega\subset\R^n$ be a bounded open set, and let $q\in C(\overline\Omega)$ satisfy $q\ge c_0>0$. Suppose that $u,v\in C(\overline\Omega)$, where $u$ is a viscosity subsolution and $v$ is a viscosity supersolution of $F_q=0$ in $\Omega$. If $u\le v$ on $\partial\Omega$, then $u\le v$ in $\overline\Omega$.
\end{proposition}

\begin{proof}
	Assume that $m=\max_{\overline\Omega}(u-v)>0$. Since $u-v\in C(\overline\Omega)$ and $\overline\Omega$ is compact, there exists $\bar x\in\overline\Omega$ such that $u(\bar x)-v(\bar x)=m$. Since $u-v\le0$ on $\partial\Omega$ and $m>0$, the point $\bar x$ cannot lie on $\partial\Omega$. Hence, $\bar x\in\Omega$. For $\delta>0$, define
	\[
	\Phi_\delta(x,y)=u(x)-v(y)-\frac{|x-y|^2}{2\delta}
	\]
	on $\overline\Omega\times\overline\Omega$. Since this set is compact and $\Phi_\delta$ is continuous, there exists $(x_\delta,y_\delta)\in\overline\Omega\times\overline\Omega$ at which 	$\Phi_\delta$ attains its maximum. Comparing with the admissible point $(\bar x,\bar x)$ gives $\Phi_\delta(x_\delta,y_\delta)\ge \Phi_\delta(\bar x,\bar x)=u(\bar x)-v(\bar x)=m$. Hence,
	\begin{equation}\label{eq:doubling-lower}
		u(x_\delta)-v(y_\delta)-\frac{|x_\delta-y_\delta|^2}{2\delta}\ge m.
	\end{equation}
	Since $u$ and $v$ are bounded, \eqref{eq:doubling-lower} implies that $|x_\delta-y_\delta|^2/\delta$ remains bounded. In particular, $|x_\delta-y_\delta|\to0$ as $\delta\downarrow0$. Passing to a subsequence, we may assume that $x_\delta\to x_0$ and $y_\delta\to x_0$ as $\delta\downarrow0$, for some $x_0\in\overline\Omega$. By \eqref{eq:doubling-lower} and continuity,
	\[
	m\le u(x_0)-v(x_0)\le\max_{\overline\Omega}(u-v)=m.
	\]
	Thus, $u(x_0)-v(x_0)=m$. Since $u\le v$ on $\partial\Omega$, it follows that $x_0\in\Omega$. We also have $x_\delta,y_\delta\in\Omega$ for all sufficiently small $\delta$.
	
	Set $p_\delta=(x_\delta-y_\delta)/\delta$. The same vector appears in the two viscosity inequalities because the derivatives of the penalization satisfy
	\[
	D_x\Big(\frac{|x-y|^2}{2\delta}\Big)\Big|_{(x_\delta,y_\delta)}=p_\delta,\quad
	-D_y\Big(\frac{|x-y|^2}{2\delta}\Big)\Big|_{(x_\delta,y_\delta)}=p_\delta.
	\]
	By the theorem of sums \cite[Theorem 3.2]{CrandallIshiiLions1992}, there are matrices $X_\delta,Y_\delta\in\operatorname{Sym}(n)$ such that $(p_\delta,X_\delta)$ belongs to the closed second-order superjet of $u$ at $x_\delta$, while $(p_\delta,Y_\delta)$ belongs to the closed second-order subjet of $v$ at $y_\delta$. These closed jets are obtained as limits of the first and second derivatives of $C^2$ test functions touching $u$ from above and $v$ from below at nearby points. Since $F_q$ is continuous, the viscosity inequalities remain valid for these limiting jets. Hence,
	\[
	\begin{aligned}
		F_q(x_\delta,u(x_\delta),p_\delta,X_\delta)
		&=-p_\delta^TX_\delta p_\delta+q(x_\delta)u(x_\delta)\le0,\\
		F_q(y_\delta,v(y_\delta),p_\delta,Y_\delta)
		&=-p_\delta^TY_\delta p_\delta+q(y_\delta)v(y_\delta)\ge0.
	\end{aligned}
	\]
	The theorem of sums also gives the matrix inequality
	\begin{equation}\label{eq:block-matrix-comparison}
		\begin{pmatrix}
			X_\delta&0\\
			0&-Y_\delta
		\end{pmatrix}
		\le\frac{3}{\delta}
		\begin{pmatrix}
			\Id&-\Id\\
			-\Id&\Id
		\end{pmatrix}.
	\end{equation}
	Testing \eqref{eq:block-matrix-comparison} against the vector $(\xi,\xi)\in\R^n\times\R^n$ gives $\xi^TX_\delta\xi-\xi^TY_\delta\xi\le0$ for every $\xi\in\R^n$, since $\begin{pmatrix}\xi\\ \xi\end{pmatrix}^{T}
	\begin{pmatrix}
		\Id&-\Id\\
		-\Id&\Id
	\end{pmatrix}
	\begin{pmatrix}\xi\\ \xi\end{pmatrix}
	=0$. It follows that $X_\delta\le Y_\delta$. Taking $\xi=p_\delta$ and using the preceding viscosity inequalities, we obtain
	\begin{equation}\label{eq:comparison-chain}
		q(x_\delta)u(x_\delta)
		\le p_\delta^TX_\delta p_\delta
		\le p_\delta^TY_\delta p_\delta
		\le q(y_\delta)v(y_\delta).
	\end{equation}
	As $\delta\downarrow0$, one has $x_\delta\to x_0$, $y_\delta\to x_0$, $u(x_\delta)\to u(x_0)$, and $v(y_\delta)\to v(x_0)$. Since $q$ is continuous, letting $\delta\downarrow0$ in \eqref{eq:comparison-chain} gives $q(x_0)u(x_0)\le q(x_0)v(x_0)$. The assumption $q(x_0)\ge c_0>0$ then implies $u(x_0)\le v(x_0)$, which contradicts $u(x_0)-v(x_0)=m>0$. Therefore, $u\le v$ in $\overline\Omega$.
	\end{proof}

\begin{proposition}[The well-posedness]\label{prop:forward-solvability}
	Let $\Omega\subset\R^n$ be a bounded open set, $q\in C(\overline\Omega)$ that satisfies $q\ge c_0>0$, and $f\in C(\partial\Omega)$. Then \eqref{eq:forward-problem} has a unique viscosity solution $u_f^q\in C(\overline\Omega)\cap\operatorname{Lip}_{\mathrm{loc}}(\Omega)$. Moreover,
	\begin{equation}\label{eq:forward-Linfty-bound}
		\|u_f^q\|_{L^\infty(\Omega)}\le\|f\|_{L^\infty(\partial\Omega)}.
	\end{equation}
	If $f\ge0$ on $\partial\Omega$, then $u_f^q\ge0$ in $\overline\Omega$.
\end{proposition}

\begin{proof}
	Set $C=\|f\|_{L^\infty(\partial\Omega)}$ and write the equation as $\Dinfty u=G(x,u)$, where $G(x,r)=q(x)r$. The constants $-C$ and $C$ are respectively a viscosity subsolution and a viscosity supersolution, since $0\ge G(x,-C)$ and $0\le G(x,C)$. They are ordered in $\overline\Omega$ and satisfy $-C\le f\le C$ on $\partial\Omega$.
	
	The function $G$ is continuous on $\overline\Omega\times\R$, nondecreasing in its second variable, and bounded on $\overline\Omega\times I$ for every bounded interval $I\subset\R$. The existence theorem for the inhomogeneous infinity Laplace Dirichlet problem can be found in \cite[Corollary 3.3]{BhattacharyaMohammed2012}. In other words, the ordered constant subsolution and supersolution satisfy the hypotheses of the Perron construction in \cite[Theorem 3.1]{BhattacharyaMohammed2012}. We obtain a viscosity solution $u_f^q\in C(\overline\Omega)$ satisfying $-C\le u_f^q\le C$ in $\overline\Omega$.
	
	The uniqueness is ensured by Proposition~\ref{prop:comparison}, and the comparison with the two constant barriers also gives
	\[
	\|u_f^q\|_{L^\infty(\Omega)}\le\|f\|_{L^\infty(\partial\Omega)}.
	\]
	Set $h(x)=q(x)u_f^q(x)$. Since $q$ and $u_f^q$ are continuous on $\overline\Omega$, one has $h\in C(\overline\Omega)\cap L^\infty(\Omega)$, and $u_f^q$ satisfies $\Dinfty u_f^q=h$ in the viscosity sense. The local regularity theorem \cite[Corollary 2.5]{BhattacharyaMohammed2012} yields $u_f^q\in\operatorname{Lip}_{\mathrm{loc}}(\Omega)$.
	
	If $f\ge0$ on $\partial\Omega$, then the zero function is a solution of the equation and lies below the boundary data. Proposition~\ref{prop:comparison} yields $u_f^q\ge0$ in $\overline\Omega$.
\end{proof}

\begin{remark}\label{rem:q-positive}
	The assumption $q\ge c_0>0$ is not an ellipticity condition. The equation is degenerate elliptic independently of the sign of $q$. The lower bound is used in the comparison argument, hence, in the uniqueness of the forward problem. If $q$ vanishes or changes sign, Proposition~\ref{prop:comparison} does not apply in its present form.
\end{remark}

We next record the boundary differentiability needed to define the DN map. A function $u\in C(\overline\Omega)$ is differentiable at $x_0\in\partial\Omega$ if there is a vector $Du(x_0)\in\R^n$ such that
\begin{equation}\label{eq:boundary-differentiability-definition}
	u(x)=u(x_0)+Du(x_0)\cdot(x-x_0)+o(|x-x_0|)
\end{equation}
as $x\to x_0$ within $\overline\Omega$. A boundary function $f\in C(\partial\Omega)$ is differentiable at $x_0$ in the tangential sense if its composition with a $C^1$ local parametrization of $\partial\Omega$ is differentiable at the parameter corresponding to $x_0$. This definition is independent of the choice of local parametrization.

\begin{proposition}[Boundary differentiability]\label{prop:boundary-differentiability}
	Assume the hypotheses of Proposition~\ref{prop:forward-solvability}. Let $x_0\in\partial\Omega$, suppose that $\partial\Omega$ is differentiable at $x_0$, and assume that $f$ is differentiable at $x_0$ in the tangential sense. Then $u_f^q$ is differentiable at $x_0$ in the sense of \eqref{eq:boundary-differentiability-definition}.
\end{proposition}

\begin{proof}
	The solution satisfies $\Dinfty u_f^q=h(x)$, where $h(x)=q(x)u_f^q(x)$. Proposition~\ref{prop:forward-solvability} gives $u_f^q\in C(\overline\Omega)$, and hence $h\in C(\Omega)\cap L^\infty(\Omega)$. The conclusion follows from the boundary differentiability theorem for inhomogeneous infinity Laplace equations \cite[Theorem 1.4]{Hong2014}.
\end{proof}

If $\partial\Omega$ is of class $C^1$ and $f\in C^1(\partial\Omega)$, Proposition~\ref{prop:boundary-differentiability} applies at every boundary point. We define
\begin{equation}\label{eq:DN-map}
	\Lambda_q(f)(x)=Du_f^q(x)\cdot\nu(x)=\partial_\nu u_f^q(x),\quad x\in\partial\Omega.
\end{equation}
If $\varphi$ is defined on $\overline\Omega$, we use $\Lambda_q(\varphi)$ as shorthand for $\Lambda_q(\varphi|_{\partial\Omega})$.

\begin{lemma}[Amplitude scaling]\label{lem:amplitude-scaling}
	Let $t>0$ and $f\in C(\partial\Omega)$. Then
	\begin{equation}\label{eq:solution-scaling}
		u_{tf}^q=t u_f^{q/t^2}.
	\end{equation}
	If $\partial\Omega$ is of class $C^1$ and $f\in C^1(\partial\Omega)$, then
	\begin{equation}\label{eq:DN-scaling}
		\Lambda_q(tf)=t\Lambda_{q/t^2}(f).
	\end{equation}
\end{lemma}

\begin{proof}
	Let $v=t^{-1}u_{tf}^q$. Since $\Dinfty(tv)=t^3\Dinfty v$, the function $v$ satisfies $-\Dinfty v+t^{-2}q(x)v=0$ in $\Omega$ and $v=f$ on $\partial\Omega$. Uniqueness gives $v=u_f^{q/t^2}$ and proves \eqref{eq:solution-scaling}. Under the additional boundary regularity, taking the normal derivative gives \eqref{eq:DN-scaling}.
\end{proof}

\section{Convex chord geometry and the rank-one correction}\label{sec:geometry}

Fix $e\in\mathbb S^{n-1}$ and let $P_e=\Id-e\otimes e$ be the orthogonal projection onto $e^\perp$. Every $x\in\R^n$ has a unique representation $x=y+se$, where $y=P_ex\in e^\perp$ and $s=e\cdot x$. We write $B_\delta^\perp(y_0)$ for the ball of radius $\delta$ in $e^\perp$.

\begin{lemma}[Geometry of parallel chords]\label{lem:chord-geometry}
	Let $\Omega\subset\R^n$ be a bounded convex domain with $C^2$ boundary and set $Y_e=P_e\Omega$.
	\begin{enumerate}[\rm(i)]
		\item The set $Y_e$ is bounded, open, and convex. For every $y\in Y_e$, there are finite numbers $\alpha_e(y)<\beta_e(y)$ such that
		\begin{equation}\label{eq:chord-parametrization}
			\Omega\cap(y+\R e)=\{y+se:\, \alpha_e(y)<s<\beta_e(y)\}.
		\end{equation}
		\item If $x_e^-(y)=y+\alpha_e(y)e$ and $x_e^+(y)=y+\beta_e(y)e$, then
		\begin{equation}\label{eq:non-glancing-endpoints}
			e\cdot\nu(x_e^+(y))>0,\quad e\cdot\nu(x_e^-(y))<0.
		\end{equation}
		\item The functions $\alpha_e,\beta_e$ belong to $C^2(Y_e)$. For every $y_0\in Y_e$, there is $\delta>0$ such that
		\begin{equation}\label{eq:local-tube-geometry}
			\Omega\cap\{y+se:y\in B_\delta^\perp(y_0)\}=\{y+se:y\in B_\delta^\perp(y_0),\ \alpha_e(y)<s<\beta_e(y)\},
		\end{equation}
		and $\beta_e-\alpha_e$ has a positive lower bound on $\overline{B_\delta^\perp(y_0)}$.
		\item On the upper and lower faces in \eqref{eq:local-tube-geometry}, the outward normals are
		\begin{equation}\label{eq:endpoint-normal-formulas}
			\nu_+(y)=\frac{e-D_y\beta_e(y)}{\sqrt{1+|D_y\beta_e(y)|^2}},\quad \nu_-(y)=\frac{D_y\alpha_e(y)-e}{\sqrt{1+|D_y\alpha_e(y)|^2}}.
		\end{equation}
		There also hold that $e\cdot\nu_+(y)=(1+|D_y\beta_e(y)|^2)^{-1/2}$ and $e\cdot\nu_-(y)=-(1+|D_y\alpha_e(y)|^2)^{-1/2}$.
	\end{enumerate}
\end{lemma}

\begin{proof}
	The projection $P_e:\R^n\to e^\perp$ is a surjective linear map and is open. Hence, $Y_e=P_e\Omega$ is open. Its boundedness and convexity follow from the corresponding properties of $\Omega$.
	
	Fix $y\in Y_e$. The intersection $\Omega\cap(y+\R e)$ is a nonempty, bounded, open, and convex subset of the affine line $y+\R e$, then it is an open interval. This gives uniquely determined finite numbers $\alpha_e(y)<\beta_e(y)$ and proves \eqref{eq:chord-parametrization}.
	
	Let $x_+=x_e^+(y)$. For every sufficiently small $h>0$, one has $x_+-he\in\Omega$. Since $\Omega$ is convex and $\partial\Omega$ is of class $C^1$, the tangent hyperplane at $x_+$ supports $\Omega$, and hence, $(z-x_+)\cdot\nu(x_+)\le0$, for every $z\in\overline\Omega$. The inequality is strict when $z\in\Omega$. Indeed, if an interior point belonged to the supporting hyperplane, a sufficiently small ball centered at that point would meet the complementary half-space. Taking $z=x_+-he$ gives $-h e\cdot\nu(x_+)<0$, then $e\cdot\nu(x_+)>0$. At the lower endpoint, the point $x_e^-(y)+he$ lies in $\Omega$ for small $h>0$, and the same argument gives $e\cdot\nu(x_e^-(y))<0$.
	
	Fix $y_0\in Y_e$ and choose $s_0$ such that $y_0+s_0e\in\Omega$. By openness, $y+s_0e\in\Omega$ for every $y$ in a sufficiently small neighborhood of $y_0$. Let $\rho_+$ and $\rho_-$ be $C^2$ defining functions in neighborhoods of $x_e^+(y_0)$ and $x_e^-(y_0)$, respectively, with $\Omega=\{\rho_\pm<0\}$ in the corresponding neighborhoods. Set $G_\pm(y,s)=\rho_\pm(y+se)$. By \eqref{eq:non-glancing-endpoints}, $\partial_sG_+(y_0,\beta_e(y_0))\ne0$ and $\partial_sG_-(y_0,\alpha_e(y_0))\ne0$. The implicit function theorem gives $C^2$ functions $\widetilde\beta$ and $\widetilde\alpha$ whose graphs describe the upper and lower boundary near the two endpoints. After shrinking the transverse neighborhood, we may assume that $\widetilde\alpha(y)<s_0<\widetilde\beta(y)$ for every $y$ in that neighborhood. Since $\Omega\cap(y+\R e)$ is an interval containing $y+s_0e$, the two graph points are precisely the lower and upper endpoints of this interval. It follows that $\widetilde\alpha=\alpha_e$ and $\widetilde\beta=\beta_e$ locally. These local representations agree on overlaps because the two endpoints of each chord are unique. Consequently, $\alpha_e,\beta_e\in C^2(Y_e)$. Since $\beta_e(y_0)-\alpha_e(y_0)>0$, continuity gives a positive lower bound for the chord length on a sufficiently small closed transverse ball.
	
	The domain lies below the upper graph $s=\beta_e(y)$ and above the lower graph $s=\alpha_e(y)$. The outward normals are therefore the normalized gradients of $s-\beta_e(y)$ and $\alpha_e(y)-s$, respectively. This gives \eqref{eq:endpoint-normal-formulas} and the stated formulas for their scalar products with $e$.
\end{proof}

\begin{lemma}\label{lem:Gamma-chords}
	Let $\Gamma_e=\{x\in\partial\Omega:\, e\cdot\nu(x)\ne0\}$. If $x\in\Gamma_e$ and $y=P_ex$, then $y\in Y_e$ and $x$ is either $x_e^-(y)$ or $x_e^+(y)$. Moreover, every endpoint of a positive length chord belongs to $\Gamma_e$.
\end{lemma}

\begin{proof}
	We first prove that every endpoint of a positive length chord belongs to $\Gamma_e$. Let $y\in Y_e$. By \eqref{eq:non-glancing-endpoints}, we know that $e\cdot\nu$ does not vanish at either endpoint, and therefore $x_e^\pm(y)\in\Gamma_e$.
	
	Conversely, let $x\in\Gamma_e$ and set $y=P_ex$. Write $s_0=e\cdot x$, so that $x=y+s_0e$. Let $\rho$ be a $C^1$ defining function for $\Omega$ in a neighborhood of $x$, chosen so that $\Omega=\{\rho<0\}$ and $D\rho(x)=|D\rho(x)|\nu(x)$. Since $x\in\Gamma_e$,
	\[
	\frac{d}{dh}\rho(x+he)\Big|_{h=0}=D\rho(x)\cdot e=|D\rho(x)|\,e\cdot\nu(x)\ne0.
	\]
	It follows that, for all sufficiently small $h>0$, the values $\rho(x+he)$ and $\rho(x-he)$ have opposite signs. Thus, exactly one of the points $x+he$ and $x-he$ lies in $\Omega$. In particular, the line $y+\R e$ contains an interior point of $\Omega$, and hence $y\in P_e\Omega=Y_e$.
	
	By \eqref{eq:chord-parametrization}, we have $\Omega\cap(y+\R e)=\{y+se:\, \alpha_e(y)<s<\beta_e(y)\}$. Since $x=y+s_0e$ belongs to $\partial\Omega$, one cannot have $\alpha_e(y)<s_0<\beta_e(y)$, for otherwise $x\in\Omega$. On the other hand, the line contains points of $\Omega$ arbitrarily close to $x$, so $s_0$ belongs to the closure of $(\alpha_e(y),\beta_e(y))$. Therefore, $s_0=\alpha_e(y)$ or $s_0=\beta_e(y)$, and hence, $x=x_e^-(y)$ or $x=x_e^+(y)$.
\end{proof}

Fix $b>R_\Omega$. Since $y\in e^\perp$, one has $\ell_{e,b}(y+se)=s+b$. For each $y\in Y_e$, consider the one-dimensional Dirichlet problem
\begin{equation}\label{eq:chord-correction}
	\begin{cases}
		\partial_s^2w_{q,e,b}(y,s)=q(y+se)(s+b) & \text{for }\alpha_e(y)<s<\beta_e(y),\\
		w_{q,e,b}(y,s)=0 & \text{for }s=\alpha_e(y)\text{ and }s=\beta_e(y).
	\end{cases}
\end{equation}
The next lemma shows that \eqref{eq:chord-correction} has a unique solution for every $y\in Y_e$. Since every $x\in\Omega$ has a unique representation $x=y+se$ with $y\in Y_e$ and $\alpha_e(y)<s<\beta_e(y)$, these chordwise solutions define a function on $\Omega$ by $w_{q,e,b}(y+se):=w_{q,e,b}(y,s)$. We continue to denote this function by $w_{q,e,b}$.

\begin{lemma}[Formula, uniqueness, and local regularity]\label{lem:correction-formula}
	For every $y\in Y_e$, the problem \eqref{eq:chord-correction} has a unique solution. Let $y_0\in Y_e$, and let $\alpha=\alpha_e$ and $\beta=\beta_e$ on a closed chord tube around $y_0$. Set $F(y,r)=q(y+re)(r+b)$ and $L(y)=\beta(y)-\alpha(y)$. Then
	\begin{equation}\label{eq:w-formula}
		w_{q,e,b}(y,s)=\int_{\alpha(y)}^s(s-r)F(y,r)\,dr-\frac{s-\alpha(y)}{L(y)}\int_{\alpha(y)}^{\beta(y)}(\beta(y)-r)F(y,r)\,dr.
	\end{equation}
	The function $w_{q,e,b}$ belongs to $C^2$ on the closed tube, vanishes on its two physical faces, and satisfies
	\begin{equation}\label{eq:w-limit-equation}
		-\partial_e^2w_{q,e,b}+q\ell_{e,b}=0.
	\end{equation}
	Moreover, $w_{q,e,b}\le0$ in $\Omega$, and
	\begin{equation}\label{eq:w-global-bound}
		\|w_{q,e,b}\|_{L^\infty(\Omega)}\le\frac{\diam(\Omega)^2}{8}(b+R_\Omega)\|q\|_{L^\infty(\Omega)}.
	\end{equation}
\end{lemma}

\begin{proof}
	Fix $y\in Y_e$ and abbreviate $\alpha=\alpha_e(y)$ and $\beta=\beta_e(y)$. Twice integrating \eqref{eq:chord-correction} from $\alpha$ gives
	\[
	w(y,s)=\int_\alpha^s(s-r)F(y,r)\,dr+C_0(y)+C_1(y)(s-\alpha).
	\]
	The boundary condition at $s=\alpha$ gives $C_0(y)=0$. Imposing the boundary condition at $s=\beta$ gives
	\[
	C_1(y)=-\frac1{\beta-\alpha}\int_\alpha^\beta(\beta-r)F(y,r)\,dr.
	\]
	This proves \eqref{eq:w-formula}. It also proves uniqueness of the chordwise Dirichlet problem.
	
	On a closed chord tube, the functions $q,\alpha,\beta$ are of class $C^2$, and $L(y)=\beta(y)-\alpha(y)$ has a positive lower bound, then the Leibniz rule may be applied to \eqref{eq:w-formula} with respect to both $y$ and $s$, including at the two physical faces. It follows that $w_{q,e,b}\in C^2$ on the closed tube. Differentiating twice with respect to $s$ gives
	\[
	\partial_s^2w_{q,e,b}(y,s)=F(y,s)=q(y+se)(s+b).
	\]
	Substituting $s=\alpha(y)$ and $s=\beta(y)$ into \eqref{eq:w-formula} gives the two zero boundary values. Thus, \eqref{eq:w-limit-equation} holds on the tube.
	
	Since $q>0$ and $s+b>0$ on every chord, one has $\partial_s^2w_{q,e,b}\ge0$. Hence, $s\mapsto w_{q,e,b}(y,s)$ is convex. A convex function that vanishes at both endpoints of an interval is nonpositive in the interval, so that $w_{q,e,b}\le0$ in $\Omega$.
	
	For the uniform bound, fix $y\in Y_e$ and continue to write $\alpha=\alpha_e(y)$ and $\beta=\beta_e(y)$. The Green function of $-\frac{d^2}{ds^2}$ on $(\alpha,\beta)$ with zero endpoint values is
	\[
	G_{\alpha,\beta}(s,r)=
	\begin{cases}
		\displaystyle\frac{(s-\alpha)(\beta-r)}{\beta-\alpha} & \text{if }\alpha\le s\le r\le\beta,\\[2mm]
		\displaystyle\frac{(r-\alpha)(\beta-s)}{\beta-\alpha} & \text{if }\alpha\le r\le s\le\beta.
	\end{cases}
	\]
	Indeed, for each fixed $r\in(\alpha,\beta)$, the function $s\mapsto G_{\alpha,\beta}(s,r)$ is affine on $(\alpha,r)$ and $(r,\beta)$, vanishes at $s=\alpha$ and $s=\beta$, and is continuous at $s=r$. Its first derivative has the jump
	\[
	\partial_sG_{\alpha,\beta}(r^+,r)-\partial_sG_{\alpha,\beta}(r^-,r)
	=-\frac{r-\alpha}{\beta-\alpha}-\frac{\beta-r}{\beta-\alpha}=-1,
	\]
	and hence, $-\partial_s^2G_{\alpha,\beta}(s,r)=\delta_r(s)$ in the distributional sense. Since $w_{q,e,b}$ satisfies $\partial_s^2w_{q,e,b}(y,s)=q(y+se)(s+b)$ with zero endpoint values, its Green representation is
	\[
	w_{q,e,b}(y,s)=-\int_\alpha^\beta G_{\alpha,\beta}(s,r)q(y+re)(r+b)\,dr.
	\]
	For fixed $s\in(\alpha,\beta)$, splitting the integral at $r=s$ gives
	\[
	\begin{aligned}
		\int_\alpha^\beta G_{\alpha,\beta}(s,r)\,dr=\frac{\beta-s}{\beta-\alpha}\int_\alpha^s(r-\alpha)\,dr
		+\frac{s-\alpha}{\beta-\alpha}\int_s^\beta(\beta-r)\,dr=\frac{(s-\alpha)(\beta-s)}{2}.
	\end{aligned}
	\]
	The last expression is maximized at $s=(\alpha+\beta)/2$, and this implies 
	\[
	\sup_{\alpha<s<\beta}\int_\alpha^\beta G_{\alpha,\beta}(s,r)\,dr
	=\frac{(\beta-\alpha)^2}{8}.
	\]
	Moreover, $y+re\in\Omega$ implies $|r|=|e\cdot(y+re)|\le R_\Omega$, so $|r+b|\le b+R_\Omega$. Since $\beta-\alpha$ is the length of the chord $\Omega\cap(y+\R e)$, one also has $\beta-\alpha\le\diam(\Omega)$. Consequently,
	\[
	\begin{aligned}
		|w_{q,e,b}(y,s)|\le \|q\|_{L^\infty(\Omega)}(b+R_\Omega)\int_\alpha^\beta G_{\alpha,\beta}(s,r)\,dr\le \frac{\diam(\Omega)^2}{8}(b+R_\Omega)\|q\|_{L^\infty(\Omega)}.
	\end{aligned}
	\]
	Taking the supremum over all chords and all points on each chord proves \eqref{eq:w-global-bound}.
\end{proof}

At every non-glancing endpoint, $w_{q,e,b}=0$ on the corresponding boundary face. Thus, its tangential derivatives vanish and
\begin{equation}\label{eq:w-normal-relation}
	\partial_e w_{q,e,b}=(e\cdot\nu)\partial_\nu w_{q,e,b}.
\end{equation}

\section{Exact large amplitude scaling and global estimates}\label{sec:global-asymptotics}

Fix $q\in C^2(\overline\Omega)$ with $q\ge c_0>0$, $e\in\mathbb S^{n-1}$, and $b>R_\Omega$. Let $u_{t,e,b}^q$ solve \eqref{eq:large-forward-intro}. For a function $\varphi$ defined on $\overline\Omega$, we write $\Lambda_q(\varphi)$ for $\Lambda_q(\varphi|_{\partial\Omega})$.

Set $\eps=t^{-2}$ and define
\begin{equation}\label{eq:v-eps-def}
	v_\eps=t^{-1}u_{t,e,b}^q.
\end{equation}
By Lemma~\ref{lem:amplitude-scaling}, $v_\eps$ is the unique solution of
\begin{equation}\label{eq:v-eps-problem}
	\begin{cases}
		-\Dinfty v_\eps+\eps q(x)v_\eps=0 & \text{in }\Omega,\\
		v_\eps=\ell_{e,b} & \text{on }\partial\Omega.
	\end{cases}
\end{equation}
Define
\begin{equation}\label{eq:z-eps-def}
	z_\eps=\frac{v_\eps-\ell_{e,b}}{\eps}.
\end{equation}
Then it is clear that $z_\eps\in C(\overline\Omega)$ and $z_\eps=0$ on $\partial\Omega$.

\begin{lemma}[Exact equation for the difference quotient]\label{lem:z-equation}
	The function $z_\eps$ is a viscosity solution of
	\begin{equation}\label{eq:z-exact}
		\mathcal F_\eps[z_\eps]=0\quad\text{in }\Omega,
	\end{equation}
	where
	\begin{equation}\label{eq:F-eps-def}
		\mathcal F_\eps[\phi]=-(e+\eps D\phi)^TD^2\phi(e+\eps D\phi)+q(x)(\ell_{e,b}+\eps\phi),\quad \phi\in C^2(\Omega).
	\end{equation}
\end{lemma}

\begin{proof}
	Suppose that $z_\eps-\phi$ has a local maximum at $x_0$. After adding a constant to the test function and continuing to denote the shifted function by $\phi$, we may assume that $z_\eps(x_0)=\phi(x_0)$. Since $v_\eps=\ell_{e,b}+\eps z_\eps$, the function $v_\eps-(\ell_{e,b}+\eps\phi)$ has a local maximum equal to zero at $x_0$. The viscosity subsolution inequality for \eqref{eq:v-eps-problem} gives
	\[
	-\Dinfty(\ell_{e,b}+\eps\phi)(x_0)+\eps q(x_0)(\ell_{e,b}(x_0)+\eps\phi(x_0))\le0.
	\]
	Since $D\ell_{e,b}=e$ and $D^2\ell_{e,b}=0$, one has
	\[
	\Dinfty(\ell_{e,b}+\eps\phi)=\eps(e+\eps D\phi)^TD^2\phi(e+\eps D\phi).
	\]
	Dividing the viscosity inequality by $\eps>0$ gives $\mathcal F_\eps[\phi](x_0)\le0$.
	
	If $z_\eps-\phi$ has a local minimum at $x_0$, then $v_\eps-(\ell_{e,b}+\eps\phi)$ has a local minimum equal to zero there. The viscosity supersolution inequality gives $\mathcal F_\eps[\phi](x_0)\ge0$. Hence, $z_\eps$ is both a viscosity subsolution and a viscosity supersolution of \eqref{eq:z-exact}.
\end{proof}

Expanding \eqref{eq:F-eps-def} in powers of $\eps$ gives
\begin{equation}\label{eq:F-eps-expansion}
	\mathcal F_\eps[\phi]=-\partial_e^2\phi+q\ell_{e,b}+\eps( q\phi-2D^2\phi[e,D\phi])-\eps^2D^2\phi[D\phi,D\phi].
\end{equation}
The leading part is the equation in \eqref{eq:w-limit-equation}.

Let $m_e=\min_{\overline\Omega}e\cdot x$ and $M_e=\max_{\overline\Omega}e\cdot x$, and define
\begin{equation}\label{eq:slab-barrier}
	h_e(x):=(e\cdot x-m_e)(M_e-e\cdot x).
\end{equation}
Then $h_e\ge0$, $D^2h_e=-2e\otimes e$, and $|\partial_eh_e|\le M_e-m_e$.

\begin{proposition}[Global bound]\label{prop:global-bound}
	There are constants $C>0$ and $\eps_0>0$, depending on $q$, $b$, $e$, and $\Omega$, such that
	\begin{equation}\label{eq:global-v-bound}
		\ell_{e,b}-C\eps h_e\le v_\eps\le\ell_{e,b}+C\eps h_e \quad \text{in }\overline{\Omega},
	\end{equation}
	for $0<\eps\le\eps_0$. Moreover,
	\begin{equation}\label{eq:global-z-bound}
		|z_\eps(x)|\le Ch_e(x)\le\frac C4(M_e-m_e)^2.
	\end{equation}
\end{proposition}

\begin{proof}
	Write $s=e\cdot x$ and set $L_e=M_e-m_e$. By the definition of $m_e$ and $M_e$, one has $m_e\le s\le M_e$ for every $x\in\overline\Omega$. Hence, $h_e(x)=(s-m_e)(M_e-s)\ge0$. A direct calculation gives $\partial_eh_e=M_e+m_e-2e\cdot x$, $Dh_e=(\partial_eh_e)e$ and $D^2h_e=-2e\otimes e$. Since $m_e\le e\cdot x\le M_e$, one has
	\begin{equation}\label{eq:slab-derivative-bound}
		|\partial_eh_e|\le M_e-m_e=L_e \quad \text{in }\overline{\Omega}.
	\end{equation}
	Notice also that $\ell_{e,b}(x)=e\cdot x+b\ge b-|x|\ge b-R_\Omega>0$ in $\overline\Omega$.
	
	Set $Q_0=\|q\|_{L^\infty(\Omega)}$, $B_0=\|\ell_{e,b}\|_{L^\infty(\Omega)}$ and choose $C=2Q_0B_0+1$. Since $\Omega$ is a nonempty open set, its width in the direction $e$ is positive, so $L_e>0$. We next choose $\eps_0=\min\{1,\frac1{2CL_e}\}$. For $0<\eps\le\eps_0$, estimate \eqref{eq:slab-derivative-bound} gives
	\begin{equation}\label{eq:slab-smallness}
		C\eps|\partial_eh_e|\le C\eps L_e\le\frac12.
	\end{equation}
	
	Define $\underline v_\eps=\ell_{e,b}-C\eps h_e$ and $\overline v_\eps=\ell_{e,b}+C\eps h_e$ to be the lower and upper barriers, respectively. We first consider the lower barrier. Since $D\ell_{e,b}=e$ and $D^2\ell_{e,b}=0$, one has $D\underline v_\eps=e-C\eps Dh_e=(1-C\eps\partial_eh_e)e$ and $D^2\underline v_\eps=-C\eps D^2h_e=2C\eps e\otimes e$. Using $|e|=1$, we obtain
	\[
	\begin{aligned}
		\Dinfty\underline v_\eps=(D\underline v_\eps)^TD^2\underline v_\eps D\underline v_\eps=\bigl((1-C\eps\partial_eh_e)e\bigr)^T(2C\eps e\otimes e)\bigl((1-C\eps\partial_eh_e)e\bigr)=2C\eps(1-C\eps\partial_eh_e)^2.
	\end{aligned}
	\]
	Therefore,
	\begin{equation}\label{eq:lower-slab-computation}
		-\Dinfty\underline v_\eps+\eps q\underline v_\eps
		=-2C\eps(1-C\eps\partial_eh_e)^2+\eps q(\ell_{e,b}-C\eps h_e).
	\end{equation}
	By \eqref{eq:slab-smallness}, $1-C\eps\partial_eh_e\ge1-C\eps|\partial_eh_e|\ge\frac12$, and hence, $-2C\eps(1-C\eps\partial_eh_e)^2\le-2C\eps\big(\frac12\big)^2=-\frac C2\eps$.
	Since $q\ge0$ and $h_e\ge0$, one also has $q(\ell_{e,b}-C\eps h_e)\le q\ell_{e,b}\le Q_0B_0$. Substitution into \eqref{eq:lower-slab-computation} yields
	\[
	\begin{aligned}
		-\Dinfty\underline v_\eps+\eps q\underline v_\eps \le-\frac C2\eps+Q_0B_0\eps=\Big(-\frac{2Q_0B_0+1}{2}+Q_0B_0\Big)\eps=-\frac{\eps}{2}<0.
	\end{aligned}
	\]
	Thus, $\underline v_\eps$ is a classical subsolution, and therefore also a viscosity subsolution, of \eqref{eq:v-eps-problem}.
	
	We next consider the upper barrier. One has
	\[
	D\overline v_\eps=e+C\eps Dh_e=(1+C\eps\partial_eh_e)e \quad \text{and}\quad 
	D^2\overline v_\eps=C\eps D^2h_e=-2C\eps e\otimes e.
	\]
	It follows that $\Dinfty\overline v_\eps=-2C\eps(1+C\eps\partial_eh_e)^2$, so
	\begin{equation}\label{eq:upper-slab-computation}
		-\Dinfty\overline v_\eps+\eps q\overline v_\eps
		=2C\eps(1+C\eps\partial_eh_e)^2+\eps q(\ell_{e,b}+C\eps h_e).
	\end{equation}
	The first term on the right-hand side is nonnegative. The second term is strictly positive because $q\ge c_0>0$, $\ell_{e,b}>0$, and $h_e\ge0$. More precisely, $q(\ell_{e,b}+C\eps h_e)\ge q\ell_{e,b}\ge c_0(b-R_\Omega)>0$. Thus, \eqref{eq:upper-slab-computation} gives $-\Dinfty\overline v_\eps+\eps q\overline v_\eps\ge\eps c_0(b-R_\Omega)>0$. It follows that $\overline v_\eps$ is a classical supersolution, and therefore also a viscosity supersolution, of \eqref{eq:v-eps-problem}.
	
	On $\partial\Omega$, the boundary condition for $v_\eps$ and the inequality $h_e\ge0$ give
	\[
	\underline v_\eps=\ell_{e,b}-C\eps h_e\le\ell_{e,b}=v_\eps\le\ell_{e,b}+C\eps h_e=\overline v_\eps.
	\]
	The potential in \eqref{eq:v-eps-problem} is $\eps q$, and it satisfies $\eps q\ge\eps c_0>0$. Proposition~\ref{prop:comparison} may therefore be applied first to $\underline v_\eps$ and $v_\eps$, and then to $v_\eps$ and $\overline v_\eps$. We obtain $\underline v_\eps\le v_\eps\le\overline v_\eps$ in $\overline\Omega$, which proves \eqref{eq:global-v-bound}. By the definition $z_\eps=(v_\eps-\ell_{e,b})/\eps$, estimate \eqref{eq:global-v-bound} gives $-Ch_e\le z_\eps\le Ch_e$. Hence, $|z_\eps|\le Ch_e$. Finally, completing the square gives $h_e(x)=\frac{(M_e-m_e)^2}{4}-\left(e\cdot x-\frac{M_e+m_e}{2}\right)^2\le\frac{(M_e-m_e)^2}{4}$. This proves \eqref{eq:global-z-bound}.
\end{proof}

\section{Local chord-tube barriers and Dirichlet-to-Neumann asymptotics}\label{sec:boundary-asymptotics}

Fix $y_0\in Y_e$. In a neighborhood of $y_0$, write $\alpha=\alpha_e$ and $\beta=\beta_e$. Choose $\delta>0$ as in Lemma~\ref{lem:chord-geometry} and set
\begin{equation}\label{eq:T-delta}
	T_\delta=\big\{y+se:\, y\in B_\delta^\perp(y_0),\ \alpha(y)<s<\beta(y)\big\}.
\end{equation}
The two physical faces are
\[
\Gamma_-^\delta=\big\{y+\alpha(y)e:\, y\in B_\delta^\perp(y_0)\big\},\quad \Gamma_+^\delta=\big\{y+\beta(y)e:\, y\in B_\delta^\perp(y_0)\big\},
\]
and the artificial lateral boundary is
\[
\Sigma_\delta=\big\{y+se:\, y\in\partial B_\delta^\perp(y_0),\ \alpha(y)\le s\le\beta(y)\big\}.
\]
Define
\begin{equation}\label{eq:eta-chi}
	\eta(y)=K_0|y-y_0|^2,\quad \chi(y,s)=(s-\alpha(y))(\beta(y)-s),
\end{equation}
then the function $\eta$ is independent of the variable $s$, while direct differentiation gives $\partial_e^2\chi=\partial_s^2\chi=-2$. Moreover, $\chi\ge0$ in $\overline{T_\delta}$ and $\chi=0$ on the two physical faces.

\begin{lemma}[Local chord-tube barrier]\label{lem:tube-barrier}
	Let $z_\eps$ be defined by \eqref{eq:z-eps-def} and let $w=w_{q,e,b}$. There are constants $K_0>0$, $A>0$, and $\eps_1>0$, independent of $\eps$, such that
	\begin{equation}\label{eq:tube-trapping}
		w-\eta-A\eps\chi\le z_\eps\le w+\eta+A\eps\chi
	\end{equation}
	in $\overline{T_\delta}$ for $0<\eps\le\eps_1$. In particular,
	\begin{equation}\label{eq:central-chord-estimate}
		|z_\eps(y_0+se)-w(y_0+se)|\le A\eps(s-\alpha_0)(\beta_0-s),
	\end{equation}
	where $\alpha_0=\alpha(y_0)$ and $\beta_0=\beta(y_0)$.
\end{lemma}

\begin{proof}
	Let $M>0$ be the uniform bound for $z_\eps$ supplied by Proposition~\ref{prop:global-bound}, so that $|z_\eps|\le M$ in $\overline\Omega$ for all sufficiently small $\eps$, and let $W=\|w\|_{L^\infty(T_\delta)}$. We first choose $K_0>0$ so that
	\begin{equation}\label{eq:K-choice}
		K_0\delta^2\ge M+W+1.
	\end{equation}
	For a constant $A>0$ to be chosen later, define
	\begin{equation}\label{eq:Phi-pm}
		\Phi_\eps^+=w+\eta+A\eps\chi,\quad \Phi_\eps^-=w-\eta-A\eps\chi.
	\end{equation}
	We will show that $\Phi_\eps^+$ is a strict supersolution and $\Phi_\eps^-$ is a strict subsolution of the exact equation \eqref{eq:z-exact}. Assume temporarily that $A\eps\le1$. For $j=0,1,2$, set
	\[
	B_j=\|D^jw\|_{L^\infty(T_\delta)}+\|D^j\eta\|_{L^\infty(T_\delta)}+\|D^j\chi\|_{L^\infty(T_\delta)},
	\]
	where $D^0$ denotes the function itself and the derivatives are taken with respect to the Euclidean variables $x=y+se$. Once $K_0$ has been fixed, the constants $B_j$ are independent of $A$ and $\eps$. Since $A\eps\le1$, one has $\|D^j\Phi_\eps^\pm\|_{L^\infty(T_\delta)}\le B_j$, $j=0,1,2$.
	
	Recall that the function $w$ satisfies $-\partial_e^2w+q\ell_{e,b}=0$, while $\partial_e^2\eta=0$ and $\partial_e^2\chi=-2$. It follows that $-\partial_e^2\Phi_\eps^++q\ell_{e,b}=2A\eps$ and $-\partial_e^2\Phi_\eps^-+q\ell_{e,b}=-2A\eps$. Using the expansion \eqref{eq:F-eps-expansion}, together with the bounds $|q\Phi_\eps^\pm|\le\|q\|_{L^\infty(T_\delta)}B_0$, $2|D^2\Phi_\eps^\pm[e,D\Phi_\eps^\pm]|\le2B_1B_2$, and $|D^2\Phi_\eps^\pm[D\Phi_\eps^\pm,D\Phi_\eps^\pm]|\le B_1^2B_2$, we obtain
	\begin{equation}\label{eq:strict-upper-lower-estimates}
		\mathcal F_\eps[\Phi_\eps^+]\ge2A\eps-C_1\eps-C_2\eps^2,\quad \mathcal F_\eps[\Phi_\eps^-]\le-2A\eps+C_1\eps+C_2\eps^2,
	\end{equation}
	where $C_1=\|q\|_{L^\infty(T_\delta)}B_0+2B_1B_2$ and $C_2=B_1^2B_2$. Now, we choose $A=C_1+C_2+1$, and then select $0<\eps_1\le\min\{\eps_0,1,A^{-1}\}$, where $\eps_0$ is the constant given in Proposition~\ref{prop:global-bound}. For $0<\eps\le\eps_1$, one has $A\eps\le1$ and $2A-C_1-C_2\eps\ge2A-C_1-C_2=C_1+C_2+2>0$. Therefore, with these choices at hand, the inequalities \eqref{eq:strict-upper-lower-estimates} read that 
	\begin{equation}\label{eq:strict-barrier-signs}
		\mathcal F_\eps[\Phi_\eps^+]>0,\quad \mathcal F_\eps[\Phi_\eps^-]<0 \quad \text{in }T_\delta.
	\end{equation}
		
	We next verify the ordering on the boundary of the tube. On the physical faces $\Gamma_-^\delta\cup\Gamma_+^\delta$, one has $z_\eps=0$, $w=0$, and $\chi=0$. Hence,
	\[
	\Phi_\eps^-=-\eta\le0=z_\eps\le\eta=\Phi_\eps^+.
	\]
	On the lateral boundary $\Sigma_\delta$, one has $\eta=K_0\delta^2$. Since $\chi\ge0$, the choice \eqref{eq:K-choice} gives
	\[
	\Phi_\eps^-\le W-K_0\delta^2\le-M-1\le z_\eps\leq M+1 \leq -W+K_0\delta^2 \leq \Phi_\eps^+.
	\]
	Thus, $\Phi_\eps^-\le z_\eps\le\Phi_\eps^+$ on $\partial T_\delta$.
	
	Suppose that $m=\max_{\overline{T_\delta}}(z_\eps-\Phi_\eps^+)>0$, and let $x_*$ be a point where the maximum is attained. The boundary ordering implies that $x_*\in T_\delta$. By the definition of $m$, one has $z_\eps-\Phi_\eps^+-m\le0$ in $\overline{T_\delta}$, with equality at $x_*$. Hence, $\Phi_\eps^++m$ touches $z_\eps$ from above at $x_*$. Since $z_\eps$ is a viscosity subsolution of \eqref{eq:z-exact}, $\mathcal F_\eps[\Phi_\eps^++m](x_*)\le0$. Adding the constant $m$ does not change the first or second derivatives, while the zeroth-order term changes by $\eps q(x)m$. Therefore,
	\[
	0\ge\mathcal F_\eps[\Phi_\eps^++m](x_*)=\mathcal F_\eps[\Phi_\eps^+](x_*)+\eps q(x_*)m>0,
	\]
	where the last inequality follows from \eqref{eq:strict-barrier-signs}, $q>0$, and $m>0$. This contradiction proves that $z_\eps\le\Phi_\eps^+$ in $\overline{T_\delta}$.
	
	Similarly, suppose that $m=\max_{\overline{T_\delta}}(\Phi_\eps^--z_\eps)>0$, and let $x^*\in\overline{T_\delta}$ be a point where this maximum is attained. The boundary ordering implies that $x^*\in T_\delta$. Then $\Phi_\eps^--m$ touches $z_\eps$ from below at $x^*$. Since $z_\eps$ is a viscosity supersolution of \eqref{eq:z-exact},
	\[
	0\le\mathcal F_\eps[\Phi_\eps^--m](x^*)=\mathcal F_\eps[\Phi_\eps^-](x^*)-\eps q(x^*)m<0,
	\]
	where the last inequality follows from \eqref{eq:strict-barrier-signs}, $q>0$, and $m>0$. This is impossible. Therefore, $\Phi_\eps^-\le z_\eps$ in $\overline{T_\delta}$, and \eqref{eq:tube-trapping} follows.
	
	On the central chord $y=y_0$, one has $\eta(y_0)=0$ and $\chi(y_0,s)=(s-\alpha_0)(\beta_0-s)$. Restricting \eqref{eq:tube-trapping} to this chord gives
	\[
	w(y_0+se)-A\eps(s-\alpha_0)(\beta_0-s)\le z_\eps(y_0+se)\le w(y_0+se)+A\eps(s-\alpha_0)(\beta_0-s),
	\]
	which is equivalent to \eqref{eq:central-chord-estimate}.
\end{proof}

\begin{remark}[Structure of the barriers]\label{rem:tube-barrier-structure}
	The transverse term $\eta$ does not contribute to the limiting operator $-\partial_e^2$ and separates the upper and lower barriers from $z_\eps$ on the artificial lateral boundary. The chord factor $\chi$ vanishes on the physical faces and satisfies $\partial_e^2\chi=-2$, which produces the strict signs in \eqref{eq:strict-barrier-signs}. No positive lower bound for the curvature of $\partial\Omega$ is used.
\end{remark}

Let $L_0=\beta_0-\alpha_0$, $x_-=y_0+\alpha_0e$, and $x_+=y_0+\beta_0e$.

\begin{proposition}[Endpoint derivative estimate]\label{prop:endpoint-derivative}
	For $0<\eps\le\eps_1$,
	\begin{equation}\label{eq:e-derivative-estimate}
		|\partial_ez_\eps(x_\pm)-\partial_ew_{q,e,b}(x_\pm)|\le A\eps L_0.
	\end{equation}
	Moreover,
	\begin{equation}\label{eq:normal-derivative-estimate}
		|\partial_\nu z_\eps(x_\pm)-\partial_\nu w_{q,e,b}(x_\pm)|\le\frac{A\eps L_0}{|e\cdot\nu(x_\pm)|}.
	\end{equation}
\end{proposition}

\begin{proof}
	The function $v_\eps$ satisfies $\Dinfty v_\eps=\eps qv_\eps$ with affine boundary values. Its right-hand side is bounded and continuous, and Proposition~\ref{prop:boundary-differentiability} shows that $v_\eps$ is differentiable at $x_-$ and $x_+$. Since $z_\eps=\frac{v_\eps-\ell_{e,b}}{\eps}$, the function $z_\eps$ is differentiable at the same points. The correction $w_{q,e,b}$ is $C^2$ near both endpoints by Lemma~\ref{lem:correction-formula}.
	
	Set $g_\eps=z_\eps-w_{q,e,b}$. Both functions vanish at $x_+$ and $x_-$, so $g_\eps(x_+)=g_\eps(x_-)=0$. For $0<h<L_0$, apply \eqref{eq:central-chord-estimate} at $x_+-he=y_0+(\beta_0-h)e$. Since $\beta_0-(\beta_0-h)=h$ and $(\beta_0-h)-\alpha_0=L_0-h$, one obtains $|g_\eps(x_+-he)|\le A\eps h(L_0-h)$. Using $g_\eps(x_+)=0$ and dividing by $h>0$ gives
    \begin{equation}\label{eq:estimate_differe_quo_g}
    	\Big|\frac{g_\eps(x_+)-g_\eps(x_+-he)}{h}\Big|\le A\eps(L_0-h).
    \end{equation}
	As $h\downarrow0$, boundary differentiability implies
	\[
	\frac{g_\eps(x_+)-g_\eps(x_+-he)}{h}\longrightarrow Dg_\eps(x_+)\cdot e=\partial_ez_\eps(x_+)-\partial_ew_{q,e,b}(x_+).
	\]
	Therefore, together with \eqref{eq:estimate_differe_quo_g}, we obtain $|\partial_ez_\eps(x_+)-\partial_ew_{q,e,b}(x_+)|\le A\eps L_0$. 
	
	Similarly, at the lower endpoint, apply \eqref{eq:central-chord-estimate} at $x_-+he=y_0+(\alpha_0+h)e$. Since $(\alpha_0+h)-\alpha_0=h$ and $\beta_0-(\alpha_0+h)=L_0-h$, one obtains $|g_\eps(x_-+he)|\le A\eps h(L_0-h)$. Using $g_\eps(x_-)=0$ gives
	\[
	\Big|\frac{g_\eps(x_-+he)-g_\eps(x_-)}{h}\Big|\le A\eps(L_0-h).
	\]
	Letting $h\downarrow0$ yields $|\partial_ez_\eps(x_-)-\partial_ew_{q,e,b}(x_-)|\le A\eps L_0$. This proves \eqref{eq:e-derivative-estimate} at both endpoints.
	
	The functions $z_\eps$ and $w_{q,e,b}$ vanish on the corresponding physical boundary faces. Their tangential derivatives therefore vanish at $x_\pm$, and their gradients are parallel to the outward unit normal. Hence,
	\[
	\partial_ez_\eps(x_\pm)=(e\cdot\nu(x_\pm))\partial_\nu z_\eps(x_\pm) \quad \text{and}\quad 
	\partial_ew_{q,e,b}(x_\pm)=(e\cdot\nu(x_\pm))\partial_\nu w_{q,e,b}(x_\pm).
	\]
	Subtracting these identities and using $e\cdot\nu(x_\pm)\ne0$ gives
	\[
	|\partial_\nu z_\eps(x_\pm)-\partial_\nu w_{q,e,b}(x_\pm)|=\frac{|\partial_ez_\eps(x_\pm)-\partial_ew_{q,e,b}(x_\pm)|}{|e\cdot\nu(x_\pm)|}.
	\]
	Estimate \eqref{eq:normal-derivative-estimate} now follows from \eqref{eq:e-derivative-estimate}.
\end{proof}

\begin{proposition}[Uniformity away from glancing points]\label{prop:local-uniformity}
	Let $K\Subset\Gamma_e$. There are $C_K>0$ and $\eps_K>0$ such that
	\begin{equation}\label{eq:uniform-normal-z}
		\sup_{x\in K}|\partial_\nu z_\eps(x)-\partial_\nu w_{q,e,b}(x)|\le C_K\eps
	\end{equation}
	for $0<\eps\le\eps_K$.
\end{proposition}

\begin{proof}
	By Lemma~\ref{lem:Gamma-chords}, every point of $K$ is an endpoint of a positive length chord. Hence, $Z=P_e(K)$ is a compact subset of the open set $Y_e$. Choose open sets $U_1,U_2\subset e^\perp$ such that $Z\Subset U_1\Subset U_2\Subset Y_e$, and set $\delta:=\frac12\operatorname{dist}(\overline U_1,e^\perp\setminus U_2)>0$, then  $\overline{B_\delta^\perp(y_0)}\subset U_2$ for every $y_0\in Z$. Since $\alpha_e,\beta_e\in C^2(Y_e)$, their $C^2$ norms are bounded on $\overline U_2$. Moreover, $\inf_{y\in\overline U_2}\bigl(\beta_e(y)-\alpha_e(y)\bigr)>0$. The explicit formula \eqref{eq:w-formula}, together with $q\in C^2(\overline\Omega)$, therefore gives a uniform $C^2$ bound for $w_{q,e,b}$ on the family of chord tubes parametrized by $\overline U_2$.
	
	Let $M$ be the global bound for $z_\eps$ from Proposition~\ref{prop:global-bound}, and let $W_2$ be a uniform bound for $w_{q,e,b}$ on these tubes. Choose $K_0$ so that $K_0\delta^2\ge M+W_2+1$. This choice is independent of the central parameter $y_0\in Z$. The $C^2$ norms of $|y-y_0|^2$ on $B_\delta^\perp(y_0)$ are also independent of $y_0$. In addition, the $C^2$ norms of $\chi_{y_0}(y,s)=(s-\alpha_e(y))(\beta_e(y)-s)$ are uniformly bounded by the $C^2$ bounds for $\alpha_e$ and $\beta_e$. It follows that the constants $B_j$, $C_1$, $C_2$, and $A$ in the proof of Lemma~\ref{lem:tube-barrier}, as well as the corresponding threshold $\eps_1$, can be chosen independently of $y_0\in Z$.
	
	The endpoint estimate \eqref{eq:e-derivative-estimate} consequently holds with one constant $A$ for every endpoint in $K$. Since $K\Subset\Gamma_e$ and the outward unit normal is continuous, one has $\kappa_K=\min_{x\in K}|e\cdot\nu(x)|>0$. Every chord has length at most $\diam(\Omega)$, then Proposition~\ref{prop:endpoint-derivative} gives 
	\[
	\sup_{x\in K}|\partial_\nu z_\eps(x)-\partial_\nu w_{q,e,b}(x)|\le\frac{A\diam(\Omega)}{\kappa_K}\eps
	\]
	for all sufficiently small $\eps$. This proves \eqref{eq:uniform-normal-z}.
\end{proof}

\begin{proof}[Proof of Theorem~\ref{thm:asymp-intro}]
	Let $\eps=t^{-2}$. Equations \eqref{eq:v-eps-def} and \eqref{eq:z-eps-def} give the exact identity
	\begin{equation}\label{eq:u-expansion-exact}
		u_{t,e,b}^q=t\ell_{e,b}+t^{-1}z_{t^{-2}}.
	\end{equation}
	All terms in \eqref{eq:u-expansion-exact} are differentiable at the boundary. Taking the outward normal derivative gives
	\begin{equation}\label{eq:DN-exact-z}
		\Lambda_q(t\ell_{e,b})=t\,e\cdot\nu+t^{-1}\partial_\nu z_{t^{-2}}.
	\end{equation}
	For $x\in\Gamma_e$, Proposition~\ref{prop:endpoint-derivative} gives $\partial_\nu z_{t^{-2}}(x)\to\partial_\nu w_{q,e,b}(x)$ as $t\to \infty$. This proves \eqref{eq:H-intro} and \eqref{eq:H-equals-normal-intro}. If $K\Subset\Gamma_e$, Proposition~\ref{prop:local-uniformity} gives $\sup_{x\in K}|\partial_\nu z_{t^{-2}}(x)-\partial_\nu w_{q,e,b}(x)|\le C_Kt^{-2}$ for all sufficiently large $t$. Substitution into \eqref{eq:DN-exact-z} proves \eqref{eq:DN-rate-intro}.
\end{proof}

\begin{corollary}\label{cor:H-coefficient}
	For every $e\in\mathbb S^{n-1}$,
	\begin{equation}\label{eq:H-local-uniform}
		t\left[\Lambda_q(t\ell_{e,b})-t\,e\cdot\nu\right]\longrightarrow\partial_\nu w_{q,e,b} \quad \text{as }t\to \infty
	\end{equation}
	locally uniformly on $\Gamma_e$.
\end{corollary}

\section{X-ray reconstruction and uniqueness}\label{sec:recovery}

We now recover the X-ray transform from the coefficient in \eqref{eq:H-intro}.

\begin{lemma}[Endpoint identity]\label{lem:endpoint-identity}
	For $e\in\mathbb S^{n-1}$ and $y\in Y_e$,
	\begin{equation}\label{eq:weighted-Xray}
		I_q(e,b;y)=\int_{\alpha_e(y)}^{\beta_e(y)}q(y+se)(s+b)\,ds.
	\end{equation}
\end{lemma}

\begin{proof}
	Fix $y\in Y_e$, and write $\alpha=\alpha_e(y)$ and $\beta=\beta_e(y)$. Define the restriction of $w_{q,e,b}$ to the corresponding chord by $W_y(s)=w_{q,e,b}(y+se),\quad \alpha\le s\le\beta$. Since $w_{q,e,b}$ is $C^2$ up to the two endpoints of the chord, the chain rule gives $W_y'(s)=Dw_{q,e,b}(y+se)\cdot e=\partial_e w_{q,e,b}(y+se)$. The chord equation reads $W_y''(s)=q(y+se)(s+b)$. Integrating it from $\alpha$ to $\beta$ gives
	\begin{equation}\label{eq:integrated-chord-equation}
		\partial_e w_{q,e,b}(x_e^+(y))-\partial_e w_{q,e,b}(x_e^-(y))
		=\int_{\alpha_e(y)}^{\beta_e(y)}q(y+se)(s+b)\,ds.
	\end{equation}
	At the two endpoints, \eqref{eq:w-normal-relation} and Theorem~\ref{thm:asymp-intro} give $\partial_e w_{q,e,b}(x_e^\pm(y))=(e\cdot\nu(x_e^\pm(y)))H_q(e,b;x_e^\pm(y))$. Substituting these identities into \eqref{eq:integrated-chord-equation} and using the definition of $I_q(e,b;y)$ proves \eqref{eq:weighted-Xray}.
\end{proof}

\begin{lemma}[Removal of the affine weight]\label{lem:remove-weight}
	For every $e\in\mathbb S^{n-1}$ and $y\in Y_e$,
	\begin{equation}\label{eq:unweighted-chord}
		\int_{\alpha_e(y)}^{\beta_e(y)}q(y+se)\,ds=\frac{I_q(e,b;y)+I_q(-e,b;y)}{2b}.
	\end{equation}
\end{lemma}

\begin{proof}
	Since $P_{-e}=P_e$, the same $y$ parametrizes the chord in both orientations. Moreover, $\alpha_{-e}(y)=-\beta_e(y)$ and $\beta_{-e}(y)=-\alpha_e(y)$. Lemma~\ref{lem:endpoint-identity} in the direction $-e$ gives
	\[
	I_q(-e,b;y)=\int_{-\beta_e(y)}^{-\alpha_e(y)}q(y-re)(r+b)\,dr.
	\]
	With the change of variables $r=-s$, this becomes
	\begin{equation}\label{eq:I-minus-e}
		I_q(-e,b;y)=\int_{\alpha_e(y)}^{\beta_e(y)}q(y+se)(b-s)\,ds.
	\end{equation}
	Adding \eqref{eq:weighted-Xray} and \eqref{eq:I-minus-e} proves \eqref{eq:unweighted-chord}.
\end{proof}

Let $Q=q\mathbf 1_\Omega$ be the zero extension of $q$ to $\R^n$. Its X-ray transform is
\begin{equation}\label{eq:Xray-def}
	\Xray Q(e,y)=\int_\R Q(y+se)\,ds,\quad e\in\mathbb S^{n-1},\quad y\in e^\perp.
\end{equation}
If $y\in Y_e$, then the line $y+\R e$ meets $\Omega$ in the chord described by \eqref{eq:chord-parametrization}, and \eqref{eq:unweighted-chord} gives the value of $\Xray Q(e,y)$. If $y\notin\overline{Y_e}$, the line does not meet $\overline\Omega$, and hence, $\Xray Q(e,y)=0$. Finally, if $y\in\partial Y_e$, the line contains no point of the open set $\Omega$. It may meet a flat portion of $\partial\Omega$, but this does not contribute to the integral of the zero extension $Q=q\mathbf 1_\Omega$. Thus, $\Xray Q(e,y)=0$ also for $y\in\partial Y_e$. It follows that \eqref{eq:Xray-reconstruction-intro} holds for every $y\in e^\perp$.

\begin{lemma}[Fourier slice identity]\label{lem:Fourier-slice}
	Let $F\in L^1_c(\R^n)$. For $e\in\mathbb S^{n-1}$ and $\xi\in e^\perp$,
	\begin{equation}\label{eq:Fourier-slice}
		\widehat F(\xi)=\int_{e^\perp}e^{-\mathsf i y\cdot\xi}\Xray F(e,y)\,dy.
	\end{equation}
	Moreover, if $\Xray F(e,y)=0$ for every $e$ and almost every $y\in e^\perp$, then $F=0$ almost everywhere.
\end{lemma}

\begin{proof}
	Write $x=y+se$, where $y\in e^\perp$ and $s\in\R$. This orthogonal change of variables has unit Jacobian, and
	\[
	\int_{e^\perp}|\Xray F(e,y)|\,dy\le\int_{e^\perp}\int_\R|F(y+se)|\,ds\,dy=\|F\|_{L^1(\R^n)}.
	\]
	Thus, Fubini's theorem applies. Since $\xi\cdot e=0$, one has $(y+se)\cdot\xi=y\cdot\xi$, and \eqref{eq:Fourier-slice} follows.
	
	If $\Xray F=0$, then $\widehat F(\xi)=0$ for every $\xi\ne0$ by choosing $e\perp\xi$. The continuity of $\widehat F$ also gives $\widehat F(0)=0$. The injectivity of the Fourier transform on $L^1(\R^n)$ yields $F=0$ almost everywhere.
\end{proof}

\begin{proposition}[Recovery of the X-ray transform]\label{prop:DN-to-Xray}
	Suppose that $\Lambda_q(t\ell_{e,b})$ is known for every $e\in\mathbb S^{n-1}$ and all sufficiently large $t$. Then the X-ray transform $\Xray Q(e,y)$ is determined for every $e\in\mathbb S^{n-1}$ and $y\in e^\perp$ by \eqref{eq:Xray-reconstruction-intro}.
\end{proposition}

\begin{proof}
	Theorem~\ref{thm:asymp-intro} determines $H_q(e,b;x)$ on $\Gamma_e$. If $y\in Y_e$, both endpoints of the corresponding chord belong to $\Gamma_e$, and hence, the measurements determine $I_q(e,b;y)$ and $I_q(-e,b;y)$. Lemma~\ref{lem:remove-weight} gives $\Xray Q(e,y)$. If $y\notin Y_e$, then $\Xray Q(e,y)=0$.
\end{proof}

\begin{proof}[Proof of Corollary~\ref{cor:reconstruction-intro}]
	Formula \eqref{eq:Xray-reconstruction-intro} follows from Proposition~\ref{prop:DN-to-Xray}. Let $\xi\ne0$ and choose $e_\xi\in\mathbb S^{n-1}$ such that $e_\xi\cdot\xi=0$. Substituting \eqref{eq:Xray-reconstruction-intro} into the Fourier slice identity \eqref{eq:Fourier-slice} gives
	\[
	\widehat Q(\xi)=\frac1{2b}\int_{e_\xi^\perp}e^{-\mathsf i y\cdot\xi}\left[I_q(e_\xi,b;y)+I_q(-e_\xi,b;y)\right]\,dy.
	\]
	The right-hand side is independent of the choice of $e_\xi$, since it equals $\widehat Q(\xi)$. The value $\widehat Q(0)$ is determined by continuity of the Fourier transform of the compactly supported $L^1$ function $Q$. Fourier inversion recovers $Q$ in $\mathcal S'(\R^n)$. Since $Q=q$ almost everywhere in $\Omega$ and $q$ is continuous, this determines $q$ pointwise in $\Omega$.
\end{proof}

\begin{proof}[Proof of Theorem~\ref{thm:main}]
	Let $Q_j=q_j\mathbf 1_\Omega$. Fix $e\in\mathbb S^{n-1}$ and $x\in\Gamma_e$. The assumed equality of the measurements gives
	\[
	t\left[\Lambda_{q_1}(t\ell_{e,b})(x)-t\,e\cdot\nu(x)\right]=t\left[\Lambda_{q_2}(t\ell_{e,b})(x)-t\,e\cdot\nu(x)\right]
	\]
	for all sufficiently large $t$. Letting $t\to\infty$ gives $H_{q_1}(e,b;x)=H_{q_2}(e,b;x)$. The same conclusion holds in the direction $-e$. Hence, $I_{q_1}(\pm e,b;y)=I_{q_2}(\pm e,b;y)$ for every $y\in Y_e$, and Lemma~\ref{lem:remove-weight} gives $\Xray Q_1(e,y)=\Xray Q_2(e,y)$. Both transforms vanish for $y\notin Y_e$. Thus, $\Xray(Q_1-Q_2)=0$, and Lemma~\ref{lem:Fourier-slice} gives $Q_1=Q_2$ almost everywhere. Since $q_1$ and $q_2$ are continuous, they agree pointwise in $\Omega$.
\end{proof}

\begin{remark}[A weaker asymptotic hypothesis]\label{rem:weaker-data}
	The proof only uses equality of the coefficients $H_{q_1}$ and $H_{q_2}$. Therefore, it is enough to assume that for every $e\in\mathbb S^{n-1}$ and every $x\in\Gamma_e$, there is a sequence $t_k\to\infty$ such that
	\[
	t_k\left[\Lambda_{q_1}(t_k\ell_{e,b})(x)-\Lambda_{q_2}(t_k\ell_{e,b})(x)\right]\to0.
	\]
\end{remark}

\begin{remark}[Reconstruction]\label{rem:constructive}
	The reconstruction proceeds by extracting $H_q$ from the large amplitude measurements, forming the endpoint quantity $I_q$, combining the directions $e$ and $-e$, and applying the inverse Fourier transform to the recovered X-ray data.
\end{remark}

\section*{Statements and declarations}

\para{Data availability statement} No datasets were generated or analyzed during the current study.

\para{Conflict of interest} The author declares no conflict of interest.

\para{Acknowledgments} The author is partially supported by the National Science and Technology Council (NSTC), Taiwan, under the project 113-2115-M-A49-017-MY3. The author also acknowledges financial support from the Alexander von Humboldt Foundation through the Henriette Herz Scouting Programme, hosted by Universit\"at Duisburg-Essen, Germany.

\bibliographystyle{alpha}
\bibliography{refs}

\end{document}